\documentclass[11pt,draft]{amsart}

\newcounter{minutes}
\divide\time by 60
\newcounter{hours}
\multiply\time by 60 \addtocounter{minutes}{-\time}

\usepackage{amsfonts}
\usepackage{amsthm}
\usepackage{amsmath}
\usepackage{graphicx}
\usepackage{amscd,amssymb,amsthm}

\newcommand{\real}{\operatorname{Re}}
\newcommand{\dt}{{\rm d}t}
\newcommand{\ds}{{\rm d}s}
\newcommand{\sign}{\operatorname{sgn}}

\title{Tur\'an type inequalities for Kr\"atzel functions}

\author[]{\'Arp\'ad Baricz$\dagger$}
\address{Department of Economics, Babe\c{s}-Bolyai University,
Cluj-Napoca 400591, Romania} \email{bariczocsi@yahoo.com}

\author[]{Dragana Jankov}
\address{Dragana Jankov : Department of Mathematics, University of Osijek, 31000 Osijek,
Croatia} \email{djankov@mathos.hr}

\author[\'A. Baricz, D. Jankov, T.K. Pog\'any]{Tibor K. Pog\'any}
\address{Faculty of Maritime Studies, University of Rijeka, Rijeka 51000, Croatia}
\email{poganj@brod.pfri.hr}

\keywords{Kr\"atzel functions, Modified Bessel functions, Tur\'an
type inequalities, Complete monotonicity, Vandermonde determinant,
Tur\'an determinant, H\"older-Rogers and Chebyshev integral
inequality.} \subjclass[2000]{Primary 33C15, Secondary 26D07.}

\thanks{$\dagger$The research
of \'A. Baricz was supported by the J\'anos Bolyai Research
Scholarship of the Hungarian Academy of Sciences and by the Romanian
National Authority for Scientific Research CNCSIS-UEFISCSU, project
number PN-II-RU-PD\underline{ }388/2010.}

\newtheorem{theorem}{Theorem}

\newtheorem*{conjecture}{Conjecture}

\begin{document}
\begin{center}
\texttt{File:~\jobname .tex,
          printed: \number\year-0\number\month-\number\day,
          \thehours.\ifnum\theminutes<10{0}\fi\theminutes}
\end{center}
\maketitle

\begin{abstract}
Complete monotonicity, Laguerre and Tur\'an type inequalities are
established for the so-called Kr\"atzel function $Z_{\rho}^{\nu},$
defined by
$$Z_{\rho}^{\nu}(u)=\int_0^{\infty}t^{\nu-1}e^{-t^{\rho}-\frac{u}{t}}\dt,$$
where $u>0$ and $\rho,\nu\in\mathbb{R}.$ Moreover, we prove the
complete monotonicity of a determinant function of which entries
involve the Kr\"atzel function.
\end{abstract}

\section{\bf Introduction}
\setcounter{equation}{0}

In 1941 while studying the zeros of Legendre polynomials, the
Hungarian mathematician Paul Tur\'an discovered the following
inequality
$$\left[P_{n}(x)\right]^2>P_{n-1}(x)P_{n+1}(x),$$
where $|x|<1,$ $n\in\{1,2,\dots\}$ and $P_n$ stands for the
classical Legendre polynomial. This inequality was published by P.
Tur\'an only in 1950 in \cite{turan}. However, since the publication
in 1948 by G. Szeg\H o \cite{szego} of the above famous Tur\'an
inequality for Legendre polynomials, many authors have deduced
analogous results for classical (orthogonal) polynomials and special
functions. In the last 62 years it has been shown by several
researchers that the most important (orthogonal) polynomials (e.g.
Laguerre, Hermite, Appell, Bernoulli, Jacobi, Jensen, Pollaczek,
Lommel, Askey-Wilson, ultraspherical polynomials) and special
functions (e.g. Bessel, modified Bessel, gamma, polygamma, Riemann
zeta functions) satisfy a Tur\'an inequality. In 1981 one of the PhD
students of P. Tur\'an, L. Alp\'ar \cite{alpar} in Tur\'an's
biography mentioned that the above Tur\'an inequality had a
wide-ranging effect, this inequality was dealt with in more than 60
papers. The Tur\'an type inequalities now have a more extensive
literature and recently some of the results have been applied
successfully in problems that arise in information theory, economic
theory and biophysics. Motivated by these applications, the Tur\'an
type inequalities have recently come under the spotlight once again
and it has been shown that, for example, the Gauss and Kummer
hypergeometric functions, as well as the generalized hypergeometric
functions, satisfy naturally some Tur\'an type inequalities. For the
most recent papers on this subject we refer to \cite{alzer},
\cite{baricz0}, \cite{baricz1}, \cite{baricz2}, \cite{baricz3},
\cite{baricz4}, \cite{baricz5}, \cite{baricz6}, \cite{baricz7},
\cite{bepe}, \cite{barnard}, \cite{dimitrov}, \cite{karp},
\cite{laforgiaism}, \cite{segura}. For more details see also the
references therein.

Motivated by the above immense research on Tur\'an type
inequalities, in this paper our aim is to deduce complete
monotonicity, lower bounds and Tur\'an type inequalities for the
so-called Kr\"atzel function, defined below. The Kr\"atzel function
is defined for $u>0,$ $\rho\in\mathbb{R}$ and $\nu\in\mathbb{C},$
being such that $\real (\nu)<0$ for $\rho\leq0,$ by the integral
\begin{equation}\label{kra}Z_{\rho}^{\nu}(u)=\int_0^{\infty}t^{\nu-1}e^{-t^{\rho}-\frac{u}{t}}\dt.\end{equation}
For $\rho\geq1$ the function \eqref{kra} was introduced by E.
Kr\"atzel \cite{kratzel} as a kernel of the integral transform
$$\left(K_{\nu}^{\rho}f\right)(u)=\int_0^{\infty}Z_{\rho}^{\nu}(ut)f(t)\dt,$$
which was applied to the solution of some ordinary differential
equations. The study of the Kr\"atzel function \eqref{kra} and the
above integral transform were continued by several authors. For
example, in \cite{kilbas2} the authors deduced explicit forms of
$Z_{\rho}^{\nu}$ in terms of the generalized Wright function, while
in \cite{kilbas} the authors obtained the asymptotic behavior of
this function at zero and infinity and gave applications to
evaluation of integrals involving $Z_{\rho}^{\nu}.$ Such
investigations now are of a great interest in connection with
applications, see \cite{kilbas} and \cite{kilbas2} and the
references therein for more details. We note that the Kr\"atzel
function occurs in the study of astrophysical thermonuclear
functions, which are derived on the basis of Boltzmann-Gibbs
statistical mechanics, see \cite{saxena}. It is also important to
note that the Kr\"atzel function $Z_1^{\nu}$ is related to the
modified Bessel function of the second kind $K_{\nu}.$ More
precisely, in view of the formula \cite[p. 237]{askey}
\begin{equation}\label{integ}K_{\nu}(u)=\frac{u^{\nu}}{2^{\nu+1}}\int_0^{\infty}t^{-\nu-1}e^{-t-\frac{u^2}{4t}}\dt\end{equation}
we have that for all $u>0$ and $\nu\in\mathbb{C}$
$$Z_1^{\nu}\left(\frac{u^2}{4}\right)=2\left(\frac{u}{2}\right)^{\nu}K_{-\nu}(u)=2\left(\frac{u}{2}\right)^{\nu}K_{\nu}(u)$$
and consequently
\begin{equation}\label{re}Z_1^{\nu}(u)=2u^{\nu/2}K_{\nu}(2\sqrt{u}).\end{equation}
Note that this function is useful in chemical physics. More
precisely, the function $$u\mapsto
2^{\nu-1}Z_1^{\nu}(uq^2/4)=(q\sqrt{u})^{\nu}K_{\nu}(q\sqrt{u})$$ is
related to the Hartree-Frock energy and is used as a basis function
for the helium isoelectronic series. See \cite{bish} and
\cite{bishop} for more details. Moreover, the function $$u\mapsto
\sqrt{2/\pi}2^{\nu-1}Z_1^{\nu}(u^2/4)=\sqrt{2/\pi}u^{\nu}K_{\nu}(u)$$
is called in the literature as reduced Bessel function and plays an
important role in theoretical chemistry. See \cite{weniger} and the
references therein for more details.

This paper is organized as follows: in the next section we present
some monotonicity, log-convexity properties, complete monotonicity
and lower bounds for the Kr\"atzel function, while in the third
section we prove the complete monotonicity of a Tur\'an determinant
which entries involve the Kr\"atzel function $Z_{\rho}^{\nu}.$ The
main result of the third section is actually a generalization of a
Tur\'an-type inequality, of which counterpart is a conjecture at the
end of this paper.

We close these preliminaries with the following definitions, which
will be used in the sequel. A function
$f:(0,\infty)\rightarrow\mathbb{R}$ is said to be completely
monotonic if $f$ has derivatives of all orders and satisfies
$$(-1)^mf^{(m)}(u)\geq 0$$ for all $u>0$ and $m\in\{0,1,\dots\}.$ A
function $g:(0,\infty)\to(0,\infty)$ is said to be logarithmically
convex, or simply log-convex, if its natural logarithm $\ln g$ is
convex, that is, for all $u_1,u_2>0$ and $\alpha\in[0,1]$ we have
$$g(\alpha u_1+(1-\alpha)u_2)\leq\left[g(u_1)\right]^{\alpha}\left[g(u_2)\right]^{1-\alpha}.$$
Note that every completely monotonic function is log-convex, see
\cite[p. 167]{widder}.

\section{\bf Monotonicity and convexity properties of the Kr\"atzel
functions} \setcounter{equation}{0}

Our first main result reads as follows:

\begin{theorem}\label{th1}
If $\nu,\rho\in\mathbb{R}$ and $u>0,$ then the following assertions
are true:
\begin{enumerate}
\item[\bf a.] The Kr\"atzel function $Z_{\rho}^{\nu}$ satisfies the recurrence relation
\begin{equation}\label{rec}\nu
Z_{\rho}^{\nu}(u)=\rho
Z_{\rho}^{\nu+\rho}(u)-uZ_\rho^{\nu-1}(u).\end{equation}
\item[\bf b.] The function $u\mapsto Z_{\rho}^{\nu}(u)$ is completely
monotonic on $(0,\infty).$
\item[\bf c.] The function $\nu\mapsto Z_{\rho}^{\nu}(u)$ is log-convex on $\mathbb{R}.$
\item[\bf d.] The function $u\mapsto Z_{\rho}^{\nu}(u)$ is log-convex on $(0,\infty).$
\item[\bf e.] For $n\in\{1,2,\dots\}$ the following Laguerre type
inequality holds
\begin{equation}\label{lag}\left[\left[Z_{\rho}^{\nu}(u)\right]^{(n)}\right]^2
\leq\left[Z_{\rho}^{\nu}(u)\right]^{(n-1)}\left[Z_{\rho}^{\nu}(u)\right]^{(n+1)}.\end{equation}
\item[\bf f.] Suppose that $\rho>0.$ Then the following inequality holds
\begin{equation}\label{bou1}
Z_{\rho}^{-\nu}(u)\geq 2u^{\frac{1}{2}-\nu}\Gamma(\nu)K_1(2\sqrt{u})
\end{equation}
provided that $\nu\geq1.$ Moreover, if $0<\nu\leq1,$ then the above
inequality is reversed. In particular, the inequality
\begin{equation}\label{bouli}u^{\nu-1}K_{\nu}(u)\geq2^{\nu-1}\Gamma(\nu)K_1(u)\end{equation}
is valid for all $\nu\geq 1.$ If $0<\nu\leq 1,$ then the above
inequality is reversed. In the above inequality equality holds if
$u$ tends to zero or $\nu=1.$
\end{enumerate}
\end{theorem}

We note here that similar result to that of \eqref{bouli} was proved
by M.E.H. Ismail \cite{ismail}. More precisely, M.E.H. Ismail proved
among others that for all $u>0$ and $\nu>1/2$ the inequality
\begin{equation}\label{eqism}
u^{\nu}K_{\nu}(u)e^{u}>2^{\nu-1}\Gamma(\nu)
\end{equation}
is valid and it is sharp as $u\to0.$ It should be mentioned here
that for $\nu\geq1$ the inequality \eqref{bouli} is better than
\eqref{eqism} since for all $\nu\geq 1$ and $u>0$ we have
$$u^{\nu}K_{\nu}(u)\geq 2^{\nu-1}\Gamma(\nu)uK_1(u)>2^{\nu-1}\Gamma(\nu)e^{-u}.$$
Observe that in fact \eqref{eqism} and the later inequality follow
from the fact (see \cite{laforgia}) that the function $$u\mapsto
u^{\nu}e^uK_{\nu}(u)$$ is strictly increasing on $(0,\infty)$ for
all $\nu>1/2$ and consequently for all $u>0$ we have
$ue^uK_{1}(u)>1.$ Here we used tacitly that when $\nu>0$ is fixed
and $u$ tends to zero, the asymptotic relation
$$u^{\nu}K_{\nu}(u)\sim 2^{\nu-1}\Gamma(\nu)$$ holds.

\begin{proof}[\bf Proof of Theorem \ref{th1}]
{\bf a.} By using \eqref{re} and the recurrence relation \cite[p.
79]{watson}
$$K_{\nu-1}(u)-K_{\nu+1}(u)=-\frac{2\nu}{u}K_{\nu}(u)$$ we obtain
$$uZ_1^{\nu-1}(u)-Z_1^{\nu+1}(u)=-\nu Z_1^{\nu}(u).$$ We note that
the above recurrence relation can be verified also by using
integration by parts as follows
\begin{align*}\nu
Z_1^{\nu}(u)&=\int_0^{\infty}\left(e^{-t-\frac{u}{t}}\right)\left(\nu
t^{\nu-1}\right)\dt\\
&=\int_0^{\infty}\left(1-\frac{u}{t^2}\right)\left(e^{-t-\frac{u}{t}}\right)t^{\nu}\dt\\&=Z_1^{\nu+1}(u)-uZ_1^{\nu-1}(u).\end{align*}
Moreover, by using the same idea we immediately have
\begin{align*}\nu
Z_{\rho}^{\nu}(u)&=\int_0^{\infty}\left(e^{-t^{\rho}-\frac{u}{t}}\right)\left(\nu
t^{\nu-1}\right)\dt\\&=\int_0^{\infty}\left(\rho
t^{\rho-1}-\frac{u}{t^2}\right)\left(e^{-t^{\rho}-\frac{u}{t}}\right)t^{\nu}\dt\\&=\rho
Z_{\rho}^{\nu+\rho}(u)-uZ_\rho^{\nu-1}(u).\end{align*}

{\bf b.} The change of variable $1/t=s$ in \eqref{kra} yields
\begin{equation}\label{krala}Z_{\rho}^{\nu}(u)=\int_0^{\infty}\left(s^{-\nu-1}e^{-s^{-\rho}}\right)e^{-us}\ds,\end{equation}
i.e. the Kr\"atzel function $Z_{\rho}^{\nu}$ is the Laplace
transform of the function $s\mapsto s^{-\nu-1}e^{-s^{-\rho}}.$ This
in view of the Bernstein-Widder theorem (see \cite{widder}) implies
that the function $u\mapsto Z_{\rho}^{\nu}(u)$ is completely
monotonic, i.e. for all $n\in\{0,1,\dots\},$ $\nu,\rho\in\mathbb{R}$
and $u>0$ we have
$$(-1)^n\left[Z_{\rho}^{\nu}(u)\right]^{(n)}>0.$$
We note that this can be verified also directly by using that for
all $n\in\{0,1,\dots\},$ $\nu,\rho\in\mathbb{R}$ and $u>0$
\begin{equation}\label{der}\left[Z_{\rho}^{\nu}(u)\right]^{(n)}=(-1)^nZ_{\rho}^{\nu-n}(u),\end{equation}
which follows via mathematical induction easily from \eqref{kra} or
\eqref{krala}.

{\bf c.} Recall the H\"older-Rogers inequality \cite[p. 54]{mitri},
that is,
\begin{equation}\label{holder}\int_a^b|f(t)g(t)|\dt\leq
{\left[\int_a^b|f(t)|^p\dt\right]}^{1/p}
{\left[\int_a^b|g(t)|^q\dt\right]}^{1/q},\end{equation} where $p>1,$
$1/p+1/q=1,$ $f$ and $g$ are real functions defined on $[a,b]$ and
$|f|^p,$ $|g|^q$ are integrable functions on $[a,b].$ Using
\eqref{kra} and \eqref{holder} we obtain that
\begin{align*}
Z_{\rho}^{\alpha\nu_1+(1-\alpha)\nu_2}(u)&=\int_0^{\infty}t^{\alpha\nu_1+(1-\alpha)\nu_2-1}e^{-t^{\rho}-\frac{u}{t}}\dt\\
&=\int_0^{\infty}t^{\alpha(\nu_1-1)+(1-\alpha)(\nu_2-1)}e^{-t^{\rho}-\frac{u}{t}}\dt\\
&=\int_0^{\infty}\left(t^{\nu_1-1}e^{-t^{\rho}-\frac{u}{t}}\right)^{\alpha}\left(t^{\nu_2-1}e^{-t^{\rho}-\frac{u}{t}}\right)^{1-\alpha}\dt\\
&\leq
\left[\int_0^{\infty}t^{\nu_1-1}e^{-t^{\rho}-\frac{u}{t}}\dt\right]^{\alpha}\left[\int_0^{\infty}t^{\nu_2-1}e^{-t^{\rho}-\frac{u}{t}}\dt\right]^{1-\alpha}\\
&=\left[Z_{\rho}^{\nu_1}(u)\right]^{\alpha}\left[Z_{\rho}^{\nu_2}(u)\right]^{1-\alpha}
\end{align*}
holds for all $\alpha\in[0,1],$ $\nu_1,\nu_2,\rho\in\mathbb{R}$ and
$u>0,$ i.e. the function $\nu\mapsto Z_{\rho}^{\nu}(u)$ is
log-convex on $\mathbb{R}.$

{\bf d.} This follows from the fact that the integrand in
\eqref{kra} or \eqref{krala} is a log-linear function of $u$ and by
using the H\"older-Rogers inequality \eqref{holder} we have that
\begin{align*}
Z_{\rho}^{\nu}(\alpha u_1+(1-\alpha) u_2)&=\int_0^{\infty}t^{\nu-1}e^{-t^{\rho}-\frac{1}{t}(\alpha u_1+(1-\alpha)u_2)}\dt\\
&=\int_0^{\infty}\left(t^{\nu-1}e^{-t^{\rho}-\frac{u_1}{t}}\right)^{\alpha}\left(t^{\nu-1}e^{-t^{\rho}-\frac{u_2}{t}}\right)^{1-\alpha}\dt\\
&\leq
\left[\int_0^{\infty}t^{\nu-1}e^{-t^{\rho}-\frac{u_1}{t}}\dt\right]^{\alpha}\left[\int_0^{\infty}t^{\nu-1}e^{-t^{\rho}-\frac{u_2}{t}}\dt\right]^{1-\alpha}\\
&=\left[Z_{\rho}^{\nu}(u_1)\right]^{\alpha}\left[Z_{\rho}^{\nu}(u_2)\right]^{1-\alpha}
\end{align*}
holds for all $\alpha\in[0,1],$ $\nu,\rho\in\mathbb{R}$ and
$u_1,u_2>0,$ i.e. the function $u\mapsto Z_{\rho}^{\nu}(u)$ is
log-convex on $(0,\infty).$

Alternatively, we may use part {\bf b} of this theorem. More
precisely, it is known that every completely monotonic function is
log-convex (see \cite[p. 167]{widder}), and then in view of part
{\bf b} the Kr\"atzel function $Z_{\rho}^{\nu}$ is log-convex on
$(0,\infty).$ Moreover, as a third proof we may use part {\bf c} of
this theorem. Namely, since the function $\nu\mapsto
Z_{\rho}^{\nu}(u)$ is log-convex, it follows that the following
Tur\'an type inequality holds for all
$\nu_1,\nu_2,\rho\in\mathbb{R}$ and $u>0$
\begin{equation}\label{tur}\left[Z_{\rho}^{\frac{\nu_1+\nu_2}{2}}(u)\right]^2\leq
Z_{\rho}^{\nu_1}(u)Z_{\rho}^{\nu_2}(u).\end{equation} Now, let
choose $\nu_1=\nu-2$ and $\nu_2=\nu,$ then we obtain the Tur\'an
type inequality
$$f_{\rho}^{\nu}(u)=\left[Z_{\rho}^{\nu-1}(u)\right]^2-Z_{\rho}^{\nu-2}(u)Z_{\rho}^{\nu}(u)\leq0,$$
which is valid for all $\nu,\rho\in\mathbb{R}$ and $u>0.$ This in
turn together with \eqref{der} implies that
$$\left[\frac{\left[Z_{\rho}^{\nu}(u)\right]'}{Z_{\rho}^{\nu}(u)}\right]'=
-\left[\frac{Z_{\rho}^{\nu-1}(u)}{Z_{\rho}^{\nu}(u)}\right]'=-\frac{f_{\rho}^{\nu}(u)}{\left[Z_{\rho}^{\nu}(u)\right]^2}\geq0,$$
i.e. the function $u\mapsto
\left[Z_{\rho}^{\nu}(u)\right]'/Z_{\rho}^{\nu}(u)$ is increasing on
$(0,\infty)$ for all $\nu,\rho\in\mathbb{R}.$

{\bf e.} This follows also from part {\bf c} of this theorem. More
precisely, in view of \eqref{der} the Laguerre type inequality
\eqref{lag} is equivalent to the Tur\'an type inequality
$$\left[Z_{\rho}^{\nu-n}(u)\right]^2\leq Z_{\rho}^{\nu-n-1}(u)Z_{\rho}^{\nu-n+1}(u),$$
which clearly follows from \eqref{tur} by choosing $\nu_1=\nu-n-1$
and $\nu_2=\nu-n+1.$

{\bf f.} Let us recall the Chebyshev integral inequality \cite[p.
40]{mitri}: If $f,g:[a,b]\rightarrow\mathbb{R}$ are integrable
functions, both increasing or both decreasing and
$p:[a,b]\rightarrow\mathbb{R}$ is a positive integrable function,
then
\begin{equation}\label{csebisev}
\int_a^bp(t)f(t)\dt\int_a^bp(t)g(t)\dt\leq
\int_a^bp(t)\dt\int_a^bp(t)f(t)g(t)\dt.
\end{equation}
Note that if one of the functions $f$ or $g$ is decreasing and the
other is increasing, then \eqref{csebisev} is reversed. We shall use
this inequality. For this by using \eqref{krala} let us write
$Z_{\rho}^{-\nu}(u)$ as
$$Z_{\rho}^{-\nu}(u)=\int_0^{\infty}e^{-ut}t^{\nu-1}e^{-t^{-\rho}}\dt$$
and let $p(t)=e^{-ut},$ $f(t)=t^{\nu-1}$ and $g(t)=e^{-t^{-\rho}}.$
Clearly $f$ is increasing (decreasing) on $(0,\infty)$ if and only
if $\nu\geq1$ ($\nu\leq1$). Since $g'(t)/g(t)=\rho t^{-\rho-1},$ it
follows that $g$ is increasing if and only if $\rho>0.$ Moreover,
$$\int_0^{\infty}p(t)\dt=\int_0^{\infty}e^{-ut}\dt=\frac{1}{u}$$ and
$$\int_0^{\infty}p(t)f(t)\dt=\int_0^{\infty}e^{-ut}t^{\nu-1}\dt=u^{-\nu}\int_0^{\infty}e^{-s}s^{\nu-1}\ds=u^{-\nu}\Gamma(\nu).$$
Similarly, integration by parts and \eqref{krala} imply
$$\int_0^{\infty}p(t)g(t)\dt=\int_0^{\infty}e^{-ut}e^{-t^{-\rho}}\dt=
\frac{\rho}{u}\int_0^{\infty}t^{-\rho-1}e^{-t^{-\rho}}e^{-ut}\dt=\frac{\rho}{u}Z_{\rho}^{\rho}(u).$$
Now, by choosing $\nu=0$ in \eqref{rec}, and using \eqref{re} we
obtain that
$$\rho Z_{\rho}^{\rho}(u)=uZ_1^{-1}(u)=2\sqrt{u}K_1(2\sqrt{u})$$ and
appealing to the Chebyshev integral inequality \eqref{csebisev} the
proof of the inequality \eqref{bou1} is done.

Finally, observe that by using the relation \cite[p. 79]{watson}
$K_{\nu}(u)=K_{-\nu}(u)$ we obtain easily
\begin{equation}\label{even}Z_1^{-\nu}(u)=u^{-\nu}Z_1^{\nu}(u),\end{equation}
and if we let $\rho=1$ in \eqref{bou1}, then by using \eqref{even}
we immediately obtain \eqref{bouli}, and with this the proof is
complete.
\end{proof}

\section{\bf Tur\'an type inequalities for Kr\"atzel functions}
\setcounter{equation}{0}

Let us consider now the Tur\'an type inequality
\begin{equation}\label{tkgnew}\left[Z_{\rho}^{\nu}(u)\right]^2-Z_{\rho}^{\nu-\rho}(u)Z_{\rho}^{\nu+\rho}(u)<0,\end{equation}
which holds for all $\nu,\rho\in\mathbb{R}$ and $u>0.$ This
inequality is actually a particular case of part {\bf c} of Theorem
\ref{th1}. More precisely, by choosing $\nu_1=\nu-\rho$ and
$\nu_2=\nu+\rho$ in \eqref{tur} the proof of \eqref{tkgnew} is done.
However, we give here an alternative proof. Just observe that
$$\left[Z_{\rho}^{\nu}(u)\right]^2 -
Z_\rho^{\nu-\rho}(u)Z_\rho^{\nu+\rho}(u)= \frac12 \int_0^\infty
\int_0^\infty (ts)^{\nu-1}e^{-t^{\rho}-s^\rho -
   u\left(\frac1{t}+\frac1s\right)}\left[ 2- \left( t/s\right)^\rho - \left( s/t\right)^\rho\right]\dt \ds$$
   and by using the elementary inequality $(t/s)^{\rho}+(s/t)^{\rho}
\ge 2$, the integrand becomes negative, which proves \eqref{tkgnew}.
Moreover the above integral representation yields the following
complete monotonicity result: the function
$$u\mapsto \left|\begin{array}{cc}Z_{\rho}^{\nu-\rho}(u)&Z_{\rho}^{\nu}(u)\\Z_{\rho}^{\nu}(u)&Z_{\rho}^{\nu+\rho}(u)\end{array}\right|=
\frac12 \int_0^\infty \int_0^\infty (ts)^{\nu-1}e^{-t^{\rho}-s^\rho
-
   u\left(\frac1{t}+\frac1s\right)}\left[\left( t/s\right)^\rho +\left( s/t\right)^\rho-2\right]\dt \ds$$
is not only positive, but even completely monotonic on $(0,\infty)$
for all $\nu,\rho\in\mathbb{R}.$ The next result is an analogue of
\cite[Theorem 2.1]{laforgiaism} for the Tur\'an determinant of
Kr\"atzel functions and provides a generalization of the above
result and of part {\bf b} of Theorem \ref{th1}. Note that in view
of \eqref{re} the following result in particular for $\rho=1$ gives
better Tur\'an type inequalities for the modified Bessel function of
the second kind $K_{\nu}$ than \cite[Theorem 2.5]{laforgiaism}. For
more details, compare the first Tur\'an type inequality in
\cite[Remark 2.6]{laforgiaism} with the right-hand side of
\eqref{tmb} below.

\begin{theorem}\label{thvan}
If $\nu,\rho\in\mathbb{R}$ and $n\in\{1,2,\dots\},$ then the
function
$$u\mapsto A_{\rho,n}^{\nu}(u)=\left|\begin{array}{cccc}Z_{\rho}^{\nu-\rho}(u)&Z_{\rho}^{\nu}(u)&\dots&Z_{\rho}^{\nu+(n-1)\rho}(u)\\
Z_{\rho}^{\nu}(u)&Z_{\rho}^{\nu+\rho}(u)&\dots&Z_{\rho}^{\nu+n\rho}(u)\\\vdots&\vdots&
&\vdots\\Z_{\rho}^{\nu+(n-1)\rho}(u)&Z_{\rho}^{\nu+n\rho}(u)&\dots&Z_{\rho}^{\nu+(2n-1)\rho}(u)\end{array}\right|$$
is completely monotonic on $(0,\infty).$
\end{theorem}

\begin{proof}[\bf Proof]
By using \eqref{kra} we have
\begin{align*}
A_{\rho,n}^{\nu}(u)&=\left|\begin{array}{cccc}Z_{\rho}^{\nu-\rho}(u)&Z_{\rho}^{\nu}(u)&\dots&Z_{\rho}^{\nu+(n-1)\rho}(u)\\
Z_{\rho}^{\nu}(u)&Z_{\rho}^{\nu+\rho}(u)&\dots&Z_{\rho}^{\nu+n\rho}(u)\\\vdots&\vdots&
&\vdots\\Z_{\rho}^{\nu+(n-1)\rho}(u)&Z_{\rho}^{\nu+n\rho}(u)&\dots&Z_{\rho}^{\nu+(2n-1)\rho}(u)\end{array}\right|\\
&=\int_{[0,\infty)^{n+1}}\left|\begin{array}{cccc}t_0^{\nu-\rho-1}&t_0^{\nu-1}&\dots&t_0^{\nu+(n-1)\rho-1}\\
t_1^{\nu-1}&t_1^{\nu+\rho-1}&\dots&t_1^{\nu+n\rho-1}\\\vdots&\vdots&
&\vdots\\t_n^{\nu+(n-1)\rho-1}&t_n^{\nu+n\rho-1}&\dots&t_n^{\nu+(2n-1)\rho-1}\end{array}\right|\prod_{j=0}^{n}e^{-t_j^{\rho}-\frac{u}{t_j}}{\rm
d}t_0{\rm d}t_1\dots{\rm d}t_n\\
&=\int_{[0,\infty)^{n+1}}\left|\begin{array}{cccc}t_{\sigma(0)}^{\nu-\rho-1}&t_{\sigma(0)}^{\nu-1}&\dots&t_{\sigma(0)}^{\nu+(n-1)\rho-1}\\
t_{\sigma(1)}^{\nu-1}&t_{\sigma(1)}^{\nu+\rho-1}&\dots&t_{\sigma(1)}^{\nu+n\rho-1}\\\vdots&\vdots&
&\vdots\\t_{\sigma(n)}^{\nu+(n-1)\rho-1}&t_{\sigma(n)}^{\nu+n\rho-1}&\dots&t_{\sigma(n)}^{\nu+(2n-1)\rho-1}\end{array}\right|\prod_{j=0}^{n}e^{-t_j^{\rho}-\frac{u}{t_j}}{\rm
d}t_0{\rm d}t_1\dots{\rm d}t_n,
\end{align*}
where $\sigma$ is a permutation on $\{0,1,\dots,n\}.$ Now, let
$\sign(\sigma)$ be denote the sign of $\sigma$ and $S_n$ be the
symmetric group on $n$ symbols. Then we obtain
\begin{align*}
A_{\rho,n}^{\nu}(u)&=\int_{[0,\infty)^{n+1}}\left|\begin{array}{cccc}t_0^{0}&t_0^{\rho}&\dots&t_0^{n\rho}\\
t_1^{\rho}&t_1^{2\rho}&\dots&t_1^{(n+1)\rho}\\\vdots&\vdots&
&\vdots\\t_n^{n\rho}&t_n^{(n+1)\rho}&\dots&t_n^{2n\rho}\end{array}\right|\sign(\sigma)\prod_{j=0}^{n}t_j^{\nu-\rho-1}e^{-t_j^{\rho}-\frac{u}{t_j}}{\rm
d}t_0{\rm d}t_1\dots{\rm d}t_n\\
&=\int_{[0,\infty)^{n+1}}\left|\begin{array}{cccc}1&t_0^{\rho}&\dots&t_0^{n\rho}\\
1&t_1^{\rho}&\dots&t_1^{n\rho}\\\vdots&\vdots&
&\vdots\\1&t_n^{\rho}&\dots&t_n^{n\rho}\end{array}\right|\sign(\sigma)\left(t_0^0t_1^{\rho}\dots
t_n^{n\rho}\right)\prod_{j=0}^{n}\left(t_j^{\nu-\rho-1}e^{-t_j^{\rho}-\frac{u}{t_j}}\right){\rm
d}t_0{\rm d}t_1\dots{\rm d}t_n\\
&=\frac{1}{(n+1)!}\int_{[0,\infty)^{n+1}}\left|\begin{array}{cccc}1&t_0^{\rho}&\dots&t_0^{n\rho}\\
1&t_1^{\rho}&\dots&t_1^{n\rho}\\\vdots&\vdots&
&\vdots\\1&t_n^{\rho}&\dots&t_n^{n\rho}\end{array}\right|\prod_{j=0}^{n}\left(t_j^{\nu-\rho-1}e^{-t_j^{\rho}-\frac{u}{t_j}}\right)\\
&\ \ \times \sum_{\sigma\in
S_{n+1}}\sign(\sigma)t_0^{\sigma(0)}(t_1^{\rho})^{\sigma(1)}\dots
(t_n^{\rho})^{\sigma(n)}{\rm
d}t_0{\rm d}t_1\dots{\rm d}t_n\\
&=\frac{1}{(n+1)!}\int_{[0,\infty)^{n+1}}
\prod_{1\leq i<j\leq n}\left(t_j^{\rho}-t_i^{\rho}\right)\prod_{j=0}^{n}\left(t_j^{\nu-\rho-1}e^{-t_j^{\rho}-\frac{u}{t_j}}\right)\\
&\ \ \times \sum_{\sigma\in
S_{n+1}}\sign(\sigma)t_0^{\sigma(0)}(t_1^{\rho})^{\sigma(1)}\dots
(t_n^{\rho})^{\sigma(n)}{\rm d}t_0{\rm d}t_1\dots{\rm d}t_n\\
&=\frac{1}{(n+1)!}\int_{[0,\infty)^{n+1}} \prod_{1\leq i<j\leq
n}\left(t_j^{\rho}-t_i^{\rho}\right)^2\prod_{j=0}^{n}\left(t_j^{\nu-\rho-1}e^{-t_j^{\rho}-\frac{u}{t_j}}\right){\rm
d}t_0{\rm d}t_1\dots{\rm d}t_n,
\end{align*}
where we used that by Leibniz's formula the Vandermonde determinant
can be written as
$$\left|\begin{array}{cccc}1&t_0^{\rho}&\dots&t_0^{n\rho}\\
1&t_1^{\rho}&\dots&t_1^{n\rho}\\\vdots&\vdots&
&\vdots\\1&t_n^{\rho}&\dots&t_n^{n\rho}\end{array}\right|=\sum_{\sigma\in
S_{n+1}}\sign(\sigma)t_0^{\sigma(0)}(t_1^{\rho})^{\sigma(1)}\dots
(t_n^{\rho})^{\sigma(n)}=\prod_{1\leq i<j\leq
n}\left(t_j^{\rho}-t_i^{\rho}\right).$$ Summarizing
$$A_{\rho,n}^{\nu}(u)=\frac{1}{(n+1)!}\int_{[0,\infty)^{n+1}} e^{-\sum\limits_{j=0}^n\left(t_j^{\rho}+\frac{u}{t_j}\right)}\prod_{1\leq i<j\leq
n}\left(t_j^{\rho}-t_i^{\rho}\right)^2\prod_{j=0}^{n}t_j^{\nu-\rho-1}{\rm
d}t_0{\rm d}t_1\dots{\rm d}t_n$$ and then for all
$\nu,\rho\in\mathbb{R},$ $u>0$ and $m\in\{0,1,\dots\}$ we obtain
\begin{align*}0<(-1)^m\left[A_{\rho,n}^{\nu}(u)\right]^{(m)}=\frac{1}{(n+1)!}\int_{[0,\infty)^{n+1}}&\left(\sum_{j=0}^n\frac{1}{t_j}\right)^m
e^{-\sum\limits_{j=0}^n\left(t_j^{\rho}+\frac{u}{t_j}\right)}\\&\times\prod_{1\leq
i<j\leq
n}\left(t_j^{\rho}-t_i^{\rho}\right)^2\prod_{j=0}^{n}t_j^{\nu-\rho-1}{\rm
d}t_0{\rm d}t_1\dots{\rm d}t_n,\end{align*} which completes the
proof.
\end{proof}

Now, let us consider the function $\Phi_{\rho}^{\nu}:(0,\infty)\to
\mathbb{R},$ defined by
$$\Phi_{\rho}^{\nu}(u)=1-\frac{Z_{\rho}^{\nu-\rho}(u)Z_{\rho}^{\nu+\rho}(u)}{\left[Z_{\rho}^{\nu}(u)\right]^2}.$$
Recall the asymptotic expansion (see \cite{kilbas,kratzel})
$$Z_{\rho}^{\nu}(u)\sim \alpha_{\rho}^{\nu}u^{\frac{2\nu-\rho}{2(\rho+1)}}e^{-\beta_{\rho}u^{\frac{\rho}{\rho+1}}},$$
where
$$\alpha_{\rho}^{\nu}=\sqrt{\frac{2\pi}{\rho+1}}\rho^{1-\frac{2\nu+1}{2(\rho+1)}}\ \
\mbox{and}\ \
\beta_{\rho}=\left(1+{1}/{\rho}\right)\rho^{\frac{1}{\rho+1}},$$
which holds for large values of $u$ and fixed $\rho>0,$
$\nu\in\mathbb{R}.$ By using the above asymptotic relation we obtain
that $\lim_{u\to\infty}\Phi_{\rho}^{\nu}(u)=0,$ which shows that in
\eqref{tkgnew} the constant $0$ is the best possible. Moreover,
based on numerical experiments we believe, but are unable to prove
the following conjecture.
\begin{conjecture}
If $\nu>\rho>0,$ then the function $\Phi_{\rho}^{\nu}$ is strictly
increasing on $(0,\infty),$ and consequently the following Tur\'an
type inequality holds
\begin{equation}\label{tkg}
\rho/(\rho-\nu)\left[Z_{\rho}^{\nu}(u)\right]^2<
\left[Z_{\rho}^{\nu}(u)\right]^2-Z_{\rho}^{\nu-\rho}(u)Z_{\rho}^{\nu+\rho}(u).
\end{equation}
\end{conjecture}

Note that for $\nu,\rho>0$ fixed if $u$ tends to zero, then the
asymptotic relation (see \cite{kilbas,kratzel}) $$\rho
Z_{\rho}^{\nu}(u)\sim \Gamma(\nu/\rho)$$ is valid. Using this
relation we obtain
$$\lim\limits_{u\to 0}\Phi_{\rho}^{\nu}(u)=1-\frac{\Gamma\left({\nu}/{\rho}-1\right)\Gamma\left({\nu}/{\rho}+1\right)}{\Gamma^2\left({\nu}/{\rho}\right)}=\rho/(\rho-\nu)$$ for all
$\nu>\rho>0,$ which shows that in \eqref{tkg} the constant
$\rho/(\rho-\nu)$ is the best possible.

It is worthwhile to note here that in fact the inequality
\eqref{tkg} is motivated by the following result. If the above
conjecture were be true then \eqref{tkg} together with
\eqref{tkgnew} would yield a generalization of \eqref{tmb}, since
for $\rho=1$ the inequalities \eqref{tkg} and \eqref{tkgnew} reduce
to \eqref{tmb}.

\begin{theorem}
Let $K_{\nu}$ be the modified Bessel function of the second kind.
Then the following Tur\'an type inequalities hold for all $\nu>1$
and $u>0$
\begin{equation}\label{tmb}
1/(1-\nu)\left[K_{\nu}(u)\right]^2<\left[K_{\nu}(u)\right]^2-K_{\nu-1}(u)K_{\nu+1}(u)<0.
\end{equation}
Moreover, the right-hand side of \eqref{tmb} holds true for all
$\nu\in\mathbb{R}.$ These inequalities are sharp in the sense that
the constants $1/(1-\nu)$ and $0$ are the best
possible.\end{theorem}

For the sake of completeness it should be mentioned that the
right-hand side of \eqref{tmb} was first proved by M.E.H. Ismail and
M.E. Muldoon \cite{muldoon}, and later by A. Laforgia and P.
Natalini \cite{natalini} and recently was deduced also by \'A.
Baricz \cite{baricz6,baricz7} and J. Segura \cite{segura}, by using
different approaches. The left-hand side of \eqref{tmb} was deduced
very recently by using completely different methods by \'A. Baricz
\cite{baricz7} and J. Segura \cite{segura}. See also \cite{bepe} for
more details on \eqref{tmb}. Note that the left-hand side of
\eqref{tmb} provides actually an upper bound for the effective
variance of the generalized Gaussian distribution. More precisely,
in \cite{lacis} the authors used (without proof) the inequality
$0<v_{\rm{eff}}<1/(\mu-1)$ for $\mu=\nu+4,$ where
$$v_{\rm{eff}}=\frac{K_{\mu-1}(u)K_{\mu+1}(u)}{\left[K_{\mu}(u)\right]^2}-1$$
is the effective variance of the generalized Gaussian distribution.

\end{document}